\documentclass[12pt, reqno]{amsart}

\usepackage[in]{fullpage}
\usepackage{parskip}
\usepackage{colonequals}
\usepackage[all,cmtip]{xy}
\usepackage{bm}
\usepackage{amssymb}
\usepackage[dvipsnames]{xcolor}
\definecolor{citecol}{RGB}{0,100,255}
\definecolor{linkcol}{RGB}{59,160,47}
\usepackage[colorlinks=true, citecolor=citecol, linkcolor=linkcol]{hyperref}

\DeclareMathOperator{\ed}{d}
\DeclareMathOperator{\dive}{div}
\DeclareMathOperator{\sign}{sign}

\newcommand{\R}{\mathbb{R}}
\newcommand{\M}{\mathcal{M}}
\newcommand{\Uu}{\mathcal{U}}
\newcommand{\N}{\mathcal{N}}
\newcommand{\Su}{\mathcal{S}}
\newcommand{\B}{\mathcal{B}}

\newcommand{\bs}{\boldsymbol}
\newcommand{\Lie}{\mathcal{L}}

\newcommand{\vertiii}[1]{{\left\vert\kern-0.25ex\left\vert\kern-0.25ex\left\vert #1 
    \right\vert\kern-0.25ex\right\vert\kern-0.25ex\right\vert}}
    
\newcommand{\restrictTo}[1]{\big|_{#1}}

\theoremstyle{plain}
\newtheorem{theorem}{Theorem}
\newtheorem*{theorem*}{Theorem}
\newtheorem{lemma}{Lemma}
\newtheorem*{lemma*}{Lemma}
\newtheorem{proposition}{Proposition}
\newtheorem*{proposition*}{Proposition}
\newtheorem{remark}{Remark}
\newtheorem*{remark*}{Remark}
\newtheorem{definition}{Definition}
\newtheorem*{definition*}{Definition}
\newtheorem{corollary}{Corollary}
\newtheorem*{corollary*}{Corollary}

\author{Yuri Bogdanskii}
\address{Igor Sikorsky Kyiv Polytechnic Institute, 37 Peremohy Avenue, 03056 Kyiv, Ukraine}
\email{y.bogdanskyi@kpi.ua}

\author{Vladyslav Shram}
\address{Igor Sikorsky Kyiv Polytechnic Institute, 37 Peremohy Avenue, 03056 Kyiv, Ukraine}
\email{shram.vladyslav@gmail.com}


\begin{document}

\title{Divergence of multivector fields on infinite-dimensional manifolds}

\begin{abstract}
This article studies divergence of multivector fields on Banach manifolds with a Radon measure. The proposed definition is consistent with the classical divergence from finite-dimensional differential geometry. Certain natural properties of divergence are transferred to the case of infinite dimension.
\end{abstract}


\maketitle

\vspace{-30pt}

\quad\\


\section{Classical divergence}

Let $\M$ be an orientable differentiable real $n$-dimensional manifold of class $C^2$. A choice of a volume form $\Omega$ on $\M$ gives rise to a divergence operator, which is defined as follows. For a vector field $\bs{X}$ (of class $C^1$), $\dive{\bs{X}}$ is a function on $\M$ such that
\begin{equation*}
\dive{\bs{X}} \cdot \Omega = \ed i_{\bs{X}} \Omega,
\end{equation*}
where $i_{\bs{X}}$ denotes the interior product of a differential form by a vector field $\bs{X}$ (Namely, $i_{\bs{X}} \omega(\bs{Z_1}, \dots, \bs{Z_{k - 1}}) = \omega(\bs{X}, \bs{Z_1}, \dots, \bs{Z_{k - 1}})$).

For a decomposable $m$-vector field $\bs{\vec{X}} = \bs{X_1} \wedge \dots \wedge \bs{X_m}$ and a differential $k$-form $\omega$, the interior product $i_{\bs{\vec{X}}} \omega = i(\bs{\vec{X}}) \omega$ of $\omega$ by $\bs{\vec{X}}$ is given by
\begin{equation}
\label{interior_product}
i_{\bs{\vec{X}}} \omega \colonequals i_{\bs{X_m}} \dots i_{\bs{X_1}} \omega, \text{ if } m \leq k,
\end{equation} 
and
\begin{equation*}
i_{\bs{\vec{X}}} \omega \colonequals 0, \text{ if } m > k.
\end{equation*}

Throughout this paper, by an $m$-vector field of class $C^p$ we mean a \textbf{linear combination of decomposable $m$-vector fields} whose components are vector fields of class $C^p$. That said, one may notice that some of the definitions and results in the article can also be transferred to multivector fields understood in a broader sense.

In an obvious way the above definition of $i_{\bs{\vec{X}}}$ extends to an arbitrary multivector field $\bs{\vec{X}}$.

This operation satisfies the following property: for any $k$-vector field $\bs{\vec{X}}$, $m$-vector field $\bs{\vec{Z}}$ and a differential $(k + m)$-form $\omega$, one has the equality
\begin{equation*}
\langle i_{\bs{\vec{X}}} \omega, \bs{\vec{Z}} \rangle = \langle \omega, \bs{\vec{X}} \wedge \bs{\vec{Z}} \rangle,
\end{equation*}
where $\langle \cdot, \cdot \rangle$ denotes the natural pairing between differential forms and multivector fields of the same degree.

Then the divergence $\dive{\bs{\vec{X}}}$ of a $k$-vector field $\bs{\vec{X}}$ is defined by the following formula (see, for example, \cite{BPH11} for an equivalent definition in terms of the Hodge operator)
\begin{equation}
\label{div_classical}
i_{\dive{\bs{\vec{X}}}} \Omega = (-1)^{k - 1} \ed i_{\bs{\vec{X}}} \Omega.
\end{equation}

\begin{remark}
In principle, we could define the interior product by a multivector field in a different way, namely $i'_{\bs{X_1} \wedge \dots \wedge \bs{X_m}} = i_{\bs{X_1}} \circ \dots \circ i_{\bs{X_m}}$. In this case, equation (\ref{div_classical}) from the definition of divergence becomes $i'_{\dive{\bs{\vec{X}}}} \Omega = \ed i'_{\bs{\vec{X}}} \Omega$. However, in this article we always use the definition of interior product $i_{\bs{\vec{X}}}$ given by (\ref{interior_product}).
\end{remark}

Existence of $\dive{\bs{\vec{X}}}$ for a multivector field $\bs{\vec{X}}$ will follow from Proposition \ref{proposition1}, and uniqueness follows from general facts of multilinear algebra (see, for example, \cite[chap.~III]{Bou70}).

Let $\M$ be a manifold of class $C^3$. Given a $(k + 1)$-vector field $\bs{\vec{X}}$ of class $C^2$ and a differential $k$-form $\omega$ of class $C_0^2$ (that is, $\omega \in C^2(\M)$ and is compactly supported) on $\M$, Stokes' theorem implies $\int\limits_\M \ed(\omega \wedge i_{\bs{\vec{X}}} \Omega) = 0$, which can be written as
\begin{equation}
\label{from_Stokes}
\int\limits_\M \ed{\omega} \wedge i_{\bs{\vec{X}}}\Omega = (-1)^{k + 1} \int\limits_\M \omega \wedge \ed i_{\bs{\vec{X}}} \Omega.
\end{equation}

\begin{lemma}
\label{lemma1}
Let $\omega$ and $\bs{\vec{X}}$ be a differential $k$-form and a $k$-vector field on $\M$, respectively. Then the following equality holds
\begin{equation}
\label{furmula_lemma1}
\omega \wedge i_{\bs{\vec{X}}} \Omega = \langle \omega, \bs{\vec{X}} \rangle \Omega.
\end{equation}
\end{lemma}
\begin{proof}
Without loss of generality we may assume that $\bs{\vec{X}}$ is decomposable: $\bs{\vec{X}} = \bs{X_1} \wedge \dots \wedge \bs{X_k}$.

We have
\begin{equation*}
\begin{split}
& \omega \wedge i_{\bs{\vec{X}}} \Omega = \omega \wedge (i_{\bs{X_k}} \dots i_{\bs{X_1}} \Omega) = (-1)^{k - 1} (i_{\bs{X_k}} \omega) \wedge (i_{\bs{X_{k - 1}}} \dots i_{\bs{X_1}} \Omega) = \dots \\
& = (-1)^{\frac{(k - 1)k}{2}} (i_{\bs{X_1}} \dots i_{\bs{X_k}} \omega) \wedge \Omega = (i_{\bs{X_k}} \dots i_{\bs{X_1}} \omega) \wedge \Omega = \langle \omega, \bs{\vec{X}} \rangle \Omega.
\end{split}
\end{equation*}
\end{proof}

Let $\mu$ be a measure on $\M$ induced by the volume form $\Omega$ (for $f \in C^1(\M)$, one has $\int\limits_\M f \,d\mu = \int\limits_\M f \Omega$). Given a differential $k$-form $\omega$ and $(k + 1)$-vector field $\bs{\vec{X}}$, using (\ref{from_Stokes}) and (\ref{furmula_lemma1}), we get
\begin{equation*}
\int\limits_\M \langle \ed{\omega}, \bs{\vec{X}} \rangle \,d\mu = \int\limits_\M \ed{\omega} \wedge i_{\bs{\vec{X}}} \Omega = (-1)^{k + 1} \int\limits_\M \omega \wedge \ed i_{\bs{\vec{X}}} \Omega = - \int\limits_\M \omega \wedge i_{\dive{\bs{\vec{X}}}} \Omega = - \int\limits_\M \langle \omega, \dive{\bs{\vec{X}}} \rangle \,d\mu.
\end{equation*}
Thus, (\ref{from_Stokes}) is equivalent to 
\begin{equation}
\label{div_measure_equiv}
\int\limits_\M \langle \ed{\omega}, \bs{\vec{X}} \rangle \,d\mu = - \int\limits_\M \langle \omega, \dive{\bs{\vec{X}}} \rangle \,d\mu.
\end{equation}

Using the measure $\mu$, one can now see the divergence of a $(k + 1)$-vector field $\bs{\vec{X}}$ on $\M$ as a $k$-vector field which satisfies (\ref{div_measure_equiv}) for any differential $k$-form of class $C_0^1$. For a manifold of class $C^3$, formula (\ref{div_measure_equiv}) leads to a definition of $\dive{\bs{\vec{X}}}$ which is equivalent to the original one.

\begin{proposition}
\label{proposition1}
Let $\bs{X}$ and $\bs{\vec{Z}}$ be a vector field and a $k$-vector field of class $C^1$ on $\M$, respectively. Then one has the following formula
\begin{equation}
\label{inductive_formula1}
\dive(\bs{X} \wedge \bs{\vec{Z}}) = \dive{\bs{X}} \cdot \bs{\vec{Z}} - \bs{X} \wedge \dive{\bs{\vec{Z}}} + \Lie_{\bs{X}}\bs{\vec{Z}}.
\end{equation}
where $\Lie_{\bs{X}}$ denotes Lie derivation along a field $\bs{X}$.
\end{proposition}
\begin{proof}
It suffices to prove formula (\ref{inductive_formula1}) only for a decomposable multivector field $\bs{\vec{Z}} = \bs{Z_1} \wedge \dots \wedge \bs{Z_k}$.

We have 
\begin{equation*}
(-1)^k \ed i_{\bs{X} \wedge \bs{\vec{Z}}} \Omega = \ed i_{\bs{\vec{Z}} \wedge \bs{X}} \Omega = \ed i_{\bs{X}} (i_{\bs{\vec{Z}}} \Omega) = - i_{\bs{X}} \ed (i_{\bs{\vec{Z}}} \Omega) + \Lie_{\bs{X}}(i_{\bs{\vec{Z}}} \Omega).
\end{equation*}

For the first term on the right-hand side we have
\begin{equation*}
- i_{\bs{X}} \ed (i_{\bs{\vec{Z}}} \Omega) = -(-1)^{k - 1} i_{\bs{X}} i_{\dive{\bs{\vec{Z}}}} \Omega = -(-1)^{k - 1} i_{\dive{\bs{\vec{Z}}} \wedge \bs{X}} \Omega = -i_{\bs{X} \wedge \dive{\bs{\vec{Z}}}} \Omega.
\end{equation*}

For the second term
\begin{equation*}
\begin{split}
& \Lie_{\bs{X}}(i_{\bs{\vec{Z}}} \Omega) = \Lie_{\bs{X}}(i_{\bs{Z_k}} \dots i_{\bs{Z_1}} \Omega) = i_{\bs{Z_k}} \Lie_{\bs{X}} (i_{\bs{Z_{k - 1}}} \dots i_{\bs{Z_1}} \Omega) + i_{\Lie_{\bs{X}} \bs{Z_k}} (i_{\bs{Z_{k - 1}}} \dots i_{\bs{Z_1}} \Omega) = \dots \\
& = i_{\bs{Z_k}} \dots i_{\bs{Z_1}} \Lie_{\bs{X}} \Omega + \sum\limits_{r = 1}^k i_{\bs{Z_k}} \dots i_{\Lie_{\bs{X}} \bs{Z_r}} \dots i_{\bs{Z_1}} \Omega = i_{\bs{\vec{Z}}} \ed i_{\bs{X}} \Omega + \sum\limits_{r = 1}^k i_{\bs{Z_1} \wedge \dots \wedge \Lie_{\bs{X}} \bs{Z_r} \wedge \dots \wedge \bs{Z_k}} \Omega \\
& = i_{\bs{\vec{Z}}} \dive{\bs{X}} \cdot \Omega + i_{\Lie_{\bs{X}} \bs{\vec{Z}}} \Omega = i_{\dive{\bs{X}} \cdot \bs{\vec{Z}}} \Omega + i_{\Lie_{\bs{X}} \bs{\vec{Z}}} \Omega = i_{\dive{\bs{X}} \cdot \bs{\vec{Z}} + \Lie_{\bs{X}} \bs{\vec{Z}}} \Omega.
\end{split}
\end{equation*}

Putting the two terms together we obtain the equality (\ref{inductive_formula1}).
\end{proof}

\begin{corollary}
Divergence of a $k$-vector field (of class $C^p$) exists and is a $(k - 1)$-vector field (of class $C^{p - 1}$).
\end{corollary}
\begin{proof}
The statement immediately follows from formula (\ref{inductive_formula1}).
\end{proof}

Given a differential $k$-form $\omega$ and a decomposable $m$-vector field $\bs{\vec{X}} = \bs{X_1} \wedge \dots \wedge \bs{X_m}$, one defines the \emph{interior product} $j_\omega \bs{\vec{X}} = j(\omega) \bs{\vec{X}}$ of $\bs{\vec{X}}$ by $\omega$ as follows
\begin{equation*}
j_\omega \bs{\vec{X}} \colonequals \frac{1}{k! (m - k)!} \sum\limits_{\sigma \in S_m} \sign(\sigma) \omega(\bs{X_{\sigma(1)}}, \dots, \bs{X_{\sigma(k)}}) \bs{X_{\sigma(k + 1)}} \wedge \dots \wedge \bs{X_{\sigma(m)}}, \text{ if } k \leq m,
\end{equation*} 
and
\begin{equation*}
j_\omega \bs{\vec{X}} \colonequals 0, \text{ if } k > m.
\end{equation*}

In an obvious way this definition extends to an arbitrary multivector field $\bs{\vec{X}}$. For a similar definition, see, for example, \cite{Mar97}.

Interior product of a multivector field by a differential form satisfies the following property: for any differential $k$-form $\omega$, differential $m$-form $\eta$ and $(k + m)$-vector field $\bs{\vec{X}}$, one has
\begin{equation*}
\langle \eta, j_\omega \bs{\vec{X}} \rangle = \langle \omega \wedge \eta, \bs{\vec{X}} \rangle.
\end{equation*}

One can prove the following generalisation of Lemma \ref{lemma1} (see \cite{BPH11}): for any differential $k$-form $\omega$ and an $m$-vector field $\bs{\vec{X}}$, the following relation holds
\begin{equation}
\label{auxiliary_formula}
i_{j(\omega) \bs{\vec{X}}} \Omega = (-1)^{k(m + 1)} \omega \wedge i_{\bs{\vec{X}}} \Omega.
\end{equation}

\begin{proposition}
Let $\omega$ and $\bs{\vec{X}}$ be a differential $k$-form and an $m$-vector field ($k < m$), respectively. Then the following Leibniz rule holds
\begin{equation*}
\dive (j(\omega) \bs{\vec{X}}) = (-1)^k j(\ed{\omega}) \bs{\vec{X}} + (-1)^k j(\omega) \dive{\bs{\vec{X}}}.
\end{equation*}
\end{proposition}
\begin{proof}
Using (\ref{auxiliary_formula}), we have
\begin{equation*}
\begin{split}
 (-1)^{m - k - 1} \ed i_{j(\omega) \bs{\vec{X}}} \Omega & = (-1)^{m - k - 1 + k(m + 1)} \ed{\omega} \wedge i_{\bs{\vec{X}}} \Omega + (-1)^{m - k - 1 + k(m + 1) + k} \omega \wedge \ed i_{\bs{\vec{X}}} \Omega \\
& = (-1)^{km + m - 1} \ed{\omega} \wedge i_{\bs{\vec{X}}} \Omega + (-1)^{km + k} \omega \wedge \ed i_{\dive{\bs{\vec{X}}}} \Omega \\
& = (-1)^{km + m - 1 + (k + 1)(m + 1)} i_{j(\ed{\omega}) \bs{\vec{X}}} \Omega + (-1)^{km + k + km} i_{j(\omega) \dive{\bs{\vec{X}}}} \Omega \\
& = (-1)^k i_{j(\ed{\omega}) \bs{\vec{X}}} \Omega + (-1)^k i_{j(\omega) \dive{\bs{\vec{X}}}} \Omega.
\end{split}
\end{equation*}
\end{proof}


\section{Associated measures on Banach manifolds (see \cite{Bog12, BM17})}
\label{associated_measures}

Let $\M$ be a connected Hausdorff real Banach manifold of class $C^2$ with a model space $E$. By a differential $k$-form on $\M$ of class $C^n$ we mean a $C^n$-section of the bundle $L_{\mathrm{alt}}^k(T\M) \to \M$, where $L_{\mathrm{alt}}^k(T\M)$ is obtained by bundling together the spaces $L_{\mathrm{alt}}^k(T_p\M)$ of all bounded alternating $k$-linear forms on $T_p\M$, so that the space $L_{\mathrm{alt}}^k(T_p\M)$ is the fibre at $p \in \M$ of this bundle.

We say that an atlas $\Omega = \{(U_\alpha, \varphi_\alpha)\}$ on $\M$ is \emph{bounded} if there exists a real number $K > 0$ such that for any pair of charts $(U_\alpha, \varphi_\alpha)$ and $(U_\beta, \varphi_\beta)$, the transition map $F_{\beta \alpha} = \varphi_\beta \circ \varphi_\alpha^{-1}$ satisfies the condition
\begin{equation*}
(x \in \varphi_\alpha(U_\alpha \cap U_\beta)) \implies (\|F_{\beta \alpha}'(x)\| \leq K, \|F_{\beta \alpha}''(x)\| \leq K).
\end{equation*}

We then say that two bounded atlases $\Omega_1$ and $\Omega_2$ are \emph{equivalent} if $\Omega_1 \cup \Omega_2$ is again a bounded atlas. A \emph{bounded structure} (of class $C^2$) on $\M$ is defined as an equivalence class of bounded atlases on $\M$.

Let $(\M_1, \Omega_1)$ and $(\M_2, \Omega_2)$ be Banach manifolds $\M_1$ and $\M_2$ of class $C^2$ modeled on $E_1$ and $E_2$ together with bounded atlases $\Omega_1$ and $\Omega_2$, respectively. We say that a map $f \colon M_1 \to \M_2$ is a \emph{bounded morphism} if there exists a real number $C > 0$ such that for any pair of charts $(U, \varphi) \in \Omega_1$ and $(V, \psi) \in \Omega_2$, the following condition is satisfied
\begin{equation*}
(p \in U, \,f(p) \in V) \implies \left(\|(\psi \circ f \circ\varphi^{-1})^{(k)}(\varphi(p))\| \leq C, ~k = 1, 2 \right).
\end{equation*}

In a natural way one then defines a \emph{bounded isomorphism} between $(\M_1, \Omega_1)$ and $(\M_2, \Omega_2)$.

The property of being a bounded morphism does not depend on the choice of representatives of the corresponding equivalence classes of bounded atlases on $\M_1$ and $\M_2$.

A choice of a bounded atlas on $\M$ leads to a well-defined notion of the length $L(\Gamma)$ of a piecewise-smooth curve $\Gamma$ in $\M$. The corresponding intrinsic metric $\rho$ is consistent with the original topology. A bounded morphism $f \colon (\M_1, \Omega_1) \to (\M_2, \Omega_2)$ is Lipschitz with respect to the corresponding intrinsic metrics.

A choice of a bounded atlas also allows to introduce a norm $\vertiii{\cdot}_p$ on the tangent space $T_p\M$ to the manifold $\M$, defined by $\vertiii{\xi}_p \colonequals \sup_\alpha \|\xi_{\varphi_{\alpha}}\|$, where $\{(U_\alpha, \varphi_\alpha)\}$ is the set of charts of the original atlas, for which $p \in U_\alpha$, and $\xi_\varphi \in E$ is the representation of a tangent vector $\xi$ in a chart $\varphi$. Furthermore, one has the property of \emph{uniform topological isomorphism} of the spaces $T_p\M$ and the model space $E$, namely $\|\xi_\varphi\| \leq \vertiii{\xi}_p \leq K \|\xi_\varphi\|$, where $K$ is the constant from the definition of a bounded atlas and $\varphi$ is a chart at the point $p \in \M$.

On a manifold with a bounded atlas $(\M, \Omega)$ one has a well-defined notion of a \emph{bounded} tensor field $\bs{T}$ of class $C^1$. One assumes that there exists a real number $C > 0$ such that for any chart $(U, \varphi)$, the local representation $\bs{T}_\varphi$ of a tensor $\bs{T}$ satisfies $\|\bs{T}_\varphi(\varphi(x))\| \leq C$ and $\|\bs{T}_\varphi'(\varphi(x))\| \leq C$ for all $x \in \varphi(U)$. Boundedness of a tensor field does not depend on the choice of a bounded atlas from the corresponding equivalence class. We say that such tensor fields are of class $C_b^1(\M)$. In a natural way we define smooth functions of class $C_b^p$ ($p = 0, 1, 2$); $C_b = C_b^0$. We will use this same notation also in the case when the domain of a field or a function is a connected open subset $V$ in $\M$, in $E$ or in the surface in $\M$. A tensor field of class $C_b^1(V)$ is said to be of class $C_0^1(V)$ if its support is bounded and contained in $V$ together with its $\varepsilon$-neighbourhood for some $\varepsilon > 0$.

We say that a bounded atlas $\Omega$ is \emph{uniform} if there exists a real number $r > 0$ such that for any $p \in \M$, there exists a chart $(U, \varphi) \in \Omega$ such that $\varphi(U)$ contains a ball of radius $r$ in $E$ centred at $\varphi(p)$. \cite{Lan99, Dal89, Bog12}

An intrinsic metric on $\M$, induced by a uniform atlas, makes $\M$ into a complete metric space. Furthermore, if a bounded atlas is equivalent to a uniform one, then the metric induced by this atlas is also complete. If an equivalence class of atlases, which defines a bounded structure on $\M$, contains a uniform atlas, we call such a structure \emph{uniform}. If manifolds $\M_1$ and $\M_2$ are boundedly isomorphic, then their structures are either both uniform or non-uniform.

The flow $\Phi(t,x)$ of a vector field $\bs{X}$ of class $C_b^1$ on a manifold $\M$ with a uniform structure is defined on $\R \times \M$. \cite[p.~92]{Lan99}

If $V$ is an open subset of $\R^m$, then, given a manifold with a bounded atlas $(\M, \Omega)$, we agree to define a bounded structure on $\M \times V$ (with a model space $E \oplus \R^m$) by the atlas $\Omega \times \mathrm{id} = \{(U \times V, \varphi \times \mathrm{id}) \colon (U, \varphi) \in \Omega\}$.

An \emph{elementary surface} $\Su \subset \M$ of codimension $m$ is defined as follows. Let $\N$ be a manifold with a bounded structure modeled on a subspace $E_1$ of $E$ of codimension $m$ (from now on we identify $E$ with $E_1 \oplus \R^m$). Let $V$ be an open neighbourhood of $\vec{0} \in \R^m$ and $g \colon \N \times V \to \Uu \subset \M$ be a bounded (straightening) isomorphism onto an open subset $\Uu$ in $\M$. Then, by definition, an elementary surface is $\Su = g(\N \times \{\vec{0}\})$.

For $\varepsilon > 0$, we define
\begin{equation*}
\Su_{-\varepsilon} \colonequals \Su \cap \{x \colon \rho(x, \M \setminus \Uu) \geq \varepsilon\}.
\end{equation*}
Then $\Su = \bigcup\limits_{n = 1}^\infty \Su_{-\frac{1}{n}}$.

We say that a differential $m$-form $\omega$ of class $C_b^1$ defined on $\Uu$ is an \emph{associated $m$-form of the embedding} $\Su \subset \M$ if for any $x \in \Su$, the tangent space $T_x\Su$ is an associated subspace of the exterior form $\omega(x)$ in $T_x\M$ (i.e. $T_x\Su = \{Y \in T_x\M \colon i_Y \omega(x) = 0\}$, where $i_Y$ is the interior product of an exterior form by a vector $Y$).

If $g \colon \N \times V \to \Uu$ is a straightening isomorphism of an elementary surface $\Su$, $P$ is a projection of $\N \times V$ onto $V$ and $h$ is a continuously differentiable function on $V$ such that $h(\vec{0}) \neq 0$, then $\omega = (g^{-1})^*P^*(h \,dt_1 \wedge \dots \wedge dt_m)$ is an example of an associated $m$-form of the embedding $\Su \subset \M$. Note that the constructed $m$-form $\omega$ is closed.

Let us now consider a Borel measure $\mu$ on $\M$. The associated measure $\sigma = \sigma_{\bs{\vec{Y}}}$ is constructed as follows.

We first consider a strictly transversal to $\Su$ system $\bs{\vec{Y}} = \{\bs{Y_1}, \dots, \bs{Y_m}\}$ of pairwise commuting vector fields of class $C_b^1$ defined on $\Uu$. Strict transversality of $\bs{\vec{Y}}$ is understood in the following sense: for each $\varepsilon > 0$, there exists $\delta > 0$ such that for any $x \in \Su_{-\varepsilon}$, one has $|\omega(\bs{\vec{Y}})(x)| = |\omega(\bs{Y_1}, \dots, \bs{Y_m})(x)| \geq \delta$. Existence of such a system of fields was proved in \cite{BM17}.

Let $\Phi_t^{\bs{Y_k}}$ denote the flow of $\bs{Y_k}$. We then define $\Phi_{\vec{t}}^{\bs{\vec{Y}}} \colonequals \Phi_{t_1}^{\bs{Y_1}} \dots \Phi_{t_m}^{\bs{Y_m}}$. One has the property $\Phi_{\vec{t} + \vec{s}}^{\bs{\vec{Y}}} = \Phi_{\vec{t}}^{\bs{\vec{Y}}} \Phi_{\vec{s}}^{\bs{\vec{Y}}}$.

For Borel sets $W \in \B(\R^m)$ and $A \in \B(\M)$, the set $\Phi_W A = \Phi_W^{\bs{\vec{Y}}} A \colonequals \{\Phi_{\vec{t}}^{\bs{\vec{Y}}}(x) \colon \vec{t} \in W, ~x \in A\}$ is Borel in $\M$. Furthermore, for each $\varepsilon > 0$, there exists $p > 0$ such that $(A \in \B(\Su_{-\varepsilon}), \,W \in \B(B_p)) \implies (\Phi_W^{\bs{\vec{Y}}}A \in \B(U))$, where $B_p = \{\vec{t} \colon \|\vec{t}\| < p\} \subset \R^m$. For any set $B \in \B(B_p)$, we define a measure $\nu_B$ on $\B(\Su_{-\varepsilon})$ by $\nu_B(A) \colonequals \mu(\Phi_B^{\bs{\vec{Y}}}A)$.

Let $\lambda_m$ denote the Lebesgue measure on $\R^m$. If for any $A \in \B(\Su_{-\varepsilon})$ the following limit exists
\begin{equation}
\label{limit}
\sigma(A) = \sigma_{\bs{\vec{Y}}}(A) = \lim_{r \to 0} \frac{\nu_{B_r}(A)}{\lambda_m(B_r)},
\end{equation}
then Nikod\'ym's theorem implies that the map $\B(\Su_{-\varepsilon}) \ni A \mapsto \sigma_{\bs{\vec{Y}}}(A) \in \R$ is a Borel measure on $\Su_{-\varepsilon}$. Writing $A \in \B(\Su)$ in the form $A = \bigcup\limits_{n = 1}^\infty (A \cap \Su_{-\frac{1}{n}})$ allows to extend the measure $\sigma_{\bs{\vec{Y}}}$ to $\B(\Su)$.

Sufficient conditions for existence of the limit (\ref{limit}) were established in \cite{BM17}; the authors suggested to call $\sigma_{\bs{\vec{Y}}}$ the \emph{surface measure} on $\Su$ of the first kind induced by the system of vector fields $\bs{\vec{Y}}$.

Throughout the remainder of this paper we always assume that the surface measure exists.

Given $\varepsilon > 0$ and $r > 0$, let $\sigma_r$ denote the measure on $\B(S_{-\varepsilon})$ defined by $\sigma_r(A) \colonequals \frac{1}{\lambda_m(B_r)} \mu(\Phi_{B_r}A)$. Then, (\ref{limit}) implies that $\sigma_r(A) \to \sigma(A)$ as $r \to 0$ for any Borel set $A \subset \Su_{-\varepsilon}$.

The following two lemmas were proved in \cite{BogNew}.

\begin{lemma}
\label{lemma2}
Suppose that $\mu$ is a Radon measure on $\M$. Then for any $\varepsilon > 0$, one has that $\sigma_r$ and $\sigma$ are Radon measures on $\Su_{-\varepsilon}$.
\end{lemma}

\begin{lemma}
\label{lemma3}
Suppose that $\mu$ is a (non-negative) Radon measure on $\M$ and $u \in C_b(\M)$. Then for any $\varepsilon > 0$ and $A \in \B(\Su_{-\varepsilon})$, the following equality holds
\begin{equation*}
\lim_{r \to 0}\frac{1}{\lambda_m(B_r)} \int\limits_{\Phi_{B_r}A} u \,d\mu = \int\limits_A u \,d\sigma.
\end{equation*}
\end{lemma}


\section{Multivector fields and divergence operator}

Let $\M$ be a Banach manifold with a bounded structure and $\mu$ be a (non-negative) Borel measure on $\M$. We say that a $k$-vector field $\bs{\vec{Z}}$ on $\M$ is \emph{$\mu$-measurable} if there exists a sequence of continuous $k$-vector fields $\bs{\vec{Z}_n}$ such that $\lim\limits_{n \to \infty} \vertiii{\bs{\vec{Z}_n}(p) - \bs{\vec{Z}}(p)}_p = 0 ~(\mathrm{mod} \,\mu)$ (here $\vertiii{\cdot}_p$ is the norm in $\bigwedge\nolimits^k(T_p\M)$, see Section \ref{associated_measures}).

For a measurable multivector field $\bs{\vec{Z}}$, the function $x \mapsto \vertiii{\bs{\vec{Z}}(x)}_x$ is $\mu$-measurable on $\M$. In the case when this function is integrable on $\M$ with respect to $\mu$ we say that $\bs{\vec{Z}}$ is \emph{integrable}: $\bs{\vec{Z}} \in L_1(\mu)$ (see \cite{BP16}). In a similar way one defines multivector fields of class $L_p(\mu)$ ($1 < p \leq \infty$).

It is easy to check that if vector fields $\bs{Z_2}, \dots, \bs{Z_k}$ are measurable and bounded on $\M$, and $\bs{Z_1}$ is a vector field of class $L_p(\mu)$, then $\bs{\vec{Z}} = \bs{Z_1} \wedge \dots \wedge \bs{Z_k} \in L_p(\mu)$. One can also prove that $(\bs{\vec{Z}} \in L_p(\mu), \,\omega \text{ is a differential } k\text{-form of class } C_b(\M)) \implies (\omega(\bs{\vec{Z}}) \in L_p(\mu))$.

Linear combinations of decomposable $k$-vector fields of class $L_p(\mu)$ form a vector space, which we will denote by $L_p \bigwedge\nolimits^k(\mu)$.

\begin{definition}
\label{div_definition}
Let $\bs{\vec{Z}} = \bs{Z_1} \wedge \dots \wedge \bs{Z_k}$ be a $k$-vector field of class $C_b^1(\M)$ (that is, $\bs{Z_i} \in C_b^1(\M)$ for $i = 1, \dots, k$). We call a $(k - 1)$-vector field $\bs{\vec{W}}$ a \emph{divergence} of $\bs{\vec{Z}}$ ($\bs{\vec{W}} = \dive{\bs{\vec{Z}}}; \bs{\vec{Z}} \in D(\dive)$) if for any differential $(k - 1)$-form $\omega \in C_0^1(\M)$ the following equality holds
\begin{equation}
\label{div_definition_equation}
\int\limits_\M \langle \omega, \bs{\vec{W}} \rangle \,d\mu = - \int\limits_\M \langle \ed{\omega}, \bs{\vec{Z}} \rangle \,d\mu.
\end{equation}  
\end{definition}

In an obvious way Definition \ref{div_definition} extends to linear combinations of decomposable multivector fields.

\begin{theorem}
\label{uniqueness_theorem}
Suppose that there exists a function of class $C_0^1$ on $E$ with a non-empty bounded support (it suffices to assume that $E$ is reflexive, see \cite{FM02}) and $\mu$ is a Radon measure. Then for any $k$-vector field $\bs{\vec{Z}}$ of class $C_b^1$, there exists no more that one element $\bs{\vec{W}} \in L_1 \bigwedge^{k - 1}(\mu)$ which satisfies Definition \ref{div_definition}.
\end{theorem}
\begin{proof}
It suffices to show that if $\bs{\vec{W}} \neq \vec{0} ~(\mathrm{mod}\,\mu)$ then there exists a $(k - 1)$-form $\omega \in C_0^1(\M)$ such that $\int\limits_\M\langle \omega, \bs{\vec{W}} \rangle \,d\mu \neq 0$.
 
\textbf{Step 1}. Since $\mu$ is Radon, there exists a compact set $L \subset \M$ with $\mu(L) > 0$ such that $\bs{\vec{W}}(x) \neq 0$ for each $x \in L$ and hence, there is a chart $\varphi \colon V \to \varphi(V) \subset E$ for which
\begin{equation}
\label{condition_step1}
\mu \left(\{ x \in V \colon \bs{\vec{W}}(x) \neq 0 \} \right) > 0.
\end{equation}

The homeomorphism $\varphi$ induces a Radon measure $\mu_\varphi$ on $\varphi(V)$ and a tensor field $\bs{\vec{W}}_\varphi$. One has $\bs{\vec{W}}_\varphi \in L_1\bigwedge\nolimits^k(\mu_\varphi)$.

\textbf{Step 2}. Let $\alpha$ be an exterior $(k - 1)$-form on $E$. Then $f \colonequals \langle \alpha, \bs{\vec{W}}_\varphi \rangle \in L_1(\mu_\varphi)$. Assuming that $\int\limits_{\varphi(V)} uf \,d\mu_\varphi = 0$ for any function $u \in C_0^1(\varphi(V))$, we will show that $f = 0 ~(\mathrm{mod} \,\mu_\varphi)$.

If $u \in C_0^1(E)$ such that $U = \{x \colon u(x) > 0 \} \neq \varnothing$ then for any function $h \in C^1(\R)$, such that $h(0) = 0$, number $k \in \R$ and vector $b \in E$ the function $v(x) = h \circ u(kx + b)$ also lies in $C_0^1(E)$. Therefore, there exists a family of functions $u_\alpha \in C_0^1(E)$ with values in $[0, 1]$ such that the sets $U_\alpha = \{x \colon u_\alpha(x) > 0\}$ form a base of the topology of $E$.

By applying Lebesgue's dominated convergence theorem, we conclude that $\int_{U_\alpha} f \,d\mu_\varphi = 0$ for any $U_\alpha$. Since the family $\{U_\alpha\}$ is closed under finite unions, for any compact set $K \subset \varphi(V)$ and $\varepsilon > 0$, there exists $U_\alpha$ such that $K \subset U_\alpha \subset K_\varepsilon$ (here and henceforth $A_\varepsilon$ denotes the $\varepsilon$-neighbourhood of a set $A$), which implies $\int_K f \,d\mu_\varphi = 0$. Since $\mu_\varphi$ is Radon, $\int_A f \,d\mu_\varphi = 0$ for any $A \in \B(\varphi(V))$, that is, $f = 0 ~(\mathrm{mod}\,\mu_\varphi)$.

\textbf{Step 3}. By applying generalised Lusin's theorem (see \cite{Fre81}) to $\bs{\vec{W}}_\varphi$ and using (\ref{condition_step1}), we get that there exists a compact set $K \subset \varphi(V)$ such that $\bs{\vec{W}}_\varphi\restrictTo{K}$ is continuous on $K$ and $\mu_\varphi \left( \{x \in K \colon \bs{\vec{W}}_\varphi(x) \neq 0\} \right) > 0$. 

The set $\bs{\vec{W}}_\varphi(K)$ lies in a separable subspace $F$ of the space $\bigwedge\nolimits^{k - 1} E$, and therefore there exists a countable family $\{\beta_n\}_{n \in \mathbb{N}}$ of exterior $(k - 1)$-forms on $E$ that separates the points of F. But Step 2 implies that $\langle \beta_n, \bs{\vec{W}}_\varphi \rangle = 0 ~(\mathrm{mod}\,\mu_\varphi)$ for all $n \in \mathbb{N}$ and hence, $\mu_\varphi \left( \{x \in K \colon \bs{\vec{W}}_\varphi(x) \neq 0 \} \right) = 0$, which is a contradiction. 
\end{proof}

\begin{proposition}
Suppose that a vector field $\bs{X}$ and $k$-vector field $\bs{\vec{Z}}$ lie in $C_b^1(\M) \cap D(\dive)$. Then $\bs{X} \wedge \bs{\vec{Z}} \in C_b^1(\M) \cap D(\dive)$ and the following equality holds 
\begin{equation}
\label{inductive_formula2}
\dive(\bs{X} \wedge \bs{\vec{Z}}) = \dive{\bs{X}} \cdot \bs{\vec{Z}} - \bs{X} \wedge \dive{\bs{\vec{Z}}} + \Lie_{\bs{X}} \bs{\vec{Z}}.
\end{equation}
\end{proposition}
\begin{proof}
Let $\omega$ be a differential $k$-form of class $C_0^1$ on $\M$. One has the equality
\begin{equation}
\label{auxiliary_equality_for_inductive_formula2}
\langle \ed{\omega}, \bs{X} \wedge \bs{\vec{Z}} \rangle = \langle i_{\bs{X}} \ed{\omega}, \bs{\vec{Z}} \rangle = \bs{X} \langle \omega, \bs{\vec{Z}} \rangle - \langle \ed i_{\bs{X}} \omega, \bs{\vec{Z}}  \rangle - \langle \omega, \Lie_{\bs{X}} \bs{\vec{Z}} \rangle.
\end{equation}

Now, by combining (\ref{div_definition_equation}) and (\ref{auxiliary_equality_for_inductive_formula2}), we get
\begin{equation*}
\int\limits_\M \langle \ed{\omega}, \bs{X} \wedge \bs{\vec{Z}} \rangle \,d\mu = - \int\limits_\M \langle \omega, -\dive{\bs{X}} \cdot \bs{\vec{Z}} + \bs{X} \wedge \dive{\bs{\vec{Z}}} - \Lie_{\bs{X}} \bs{\vec{Z}} \rangle \,d\mu,
\end{equation*}
which proves the proposition. 
\end{proof}

\begin{corollary}
If $\bs{\vec{Z}} = \bs{Z_1} \wedge \dots \wedge \bs{Z_k}$ and all $\bs{Z_i} \in C_b^1(\M) \cap D(\dive)$, then $\bs{\vec{Z}} \in C_b^1(\M) \cap D(\dive)$.
\end{corollary}

\begin{proposition}
Suppose that an $m$-vector field $\bs{\vec{Z}}$ lies in $C_b^1(\M) \cap D(\dive)$ and let $\omega$ be a differential $k$-form ($k < m$) of class $C_b^1(\M)$. Then, $j(\omega) \bs{\vec{Z}}$ also lies in $C_b^1(\M) \cap D(\dive)$ and the following Leibniz rule holds
\begin{equation*}
\dive(j(\omega) \bs{\vec{Z}}) = (-1)^k j(\ed{\omega}) \bs{\vec{Z}} + (-1)^k j(\omega) \dive{\bs{\vec{Z}}}.
\end{equation*}
\end{proposition}
\begin{proof}
For any differential $(m - k - 1)$-form $\eta$ of class $C_0^1(\M)$, we have
\begin{equation*}
\begin{split}
& \int\limits_{\mathcal{M}} \left( \left\langle \ed \eta, j(\omega) \bs{\vec{Z}} \right\rangle + \left\langle \eta, (-1)^k j(\ed{\omega}) \bs{\vec{Z}} + (-1)^k j(\omega) \dive{\bs{\vec{Z}}} \right\rangle \right) \, d \mu\\
& = \int\limits_{\mathcal{M}} \left( \langle \omega \wedge \ed \eta, \bs{\vec{Z}} \rangle + (-1)^k \langle \ed \omega \wedge \eta, \bs{\vec{Z}} \rangle + (-1)^k \langle \omega \wedge \eta, \dive{\bs{\vec{Z}}} \rangle  \right) \, d \mu\\
& = \int\limits_{\mathcal{M}} \left( (-1)^k \langle \ed (\omega \wedge \eta), \bs{\vec{Z}} \rangle + (-1)^k \langle \omega \wedge \eta, \dive{\bs{\vec{Z}}} \rangle  \right) \, d \mu = 0.
\end{split}
\end{equation*}
\end{proof}


\section{Divergence on submanifolds}

If $\M$ is a finite-dimensional (oriented) manifold endowed with a volume form $\Omega$, and $\Uu$ is its open submanifold, then it is natural to take $\Omega\restrictTo{\Uu}$ to be the volume form on $\Uu$. In this case one has the equality
\begin{equation}
\label{restriction_furmula}
\dive_\Uu(\bs{\vec{Z}}\restrictTo{\Uu}) = (\dive{\bs{\vec{Z}}})\restrictTo{\Uu},
\end{equation}
where $\dive_\Uu$ is the divergence on $\Uu$, induced by the volume form $\Omega\restrictTo{\Uu}$.

In the case when $\Uu$ is an open submanifold of a Banach manifold $\M$, the definition of divergence $\dive_\Uu$ of a multivector field is obtained from Definition \ref{div_definition} by replacing (\ref{div_definition_equation}) with
\begin{equation*}
\int\limits_\Uu \langle \omega, \bs{\vec{W}} \rangle \,d\mu = - \int\limits_\Uu \langle \ed{\omega}, \bs{\vec{Z}} \rangle \,d\mu,
\end{equation*}
which now has to hold for any differential form of class $C_0^1(\Uu)$. In this case formula (\ref{restriction_furmula}) also holds.

Let now $\M$ be an orientable manifold of finite dimension $n$; $\Su \subset \M$ an orientable embedded submanifold of dimension $m = n - p$, which is an elementary surface in the sense of Section \ref{associated_measures}; $\alpha$ an associated differential $p$-form of the embedding $\Su \subset \M$; $\bs{\vec{Y}} = \{\bs{Y_1}, \dots, \bs{Y_p}\}$ a commuting strictly transversal to $\Su$ system of vector fields of class $C_b^1(\Uu)$, where $\Uu$ is from the definition of an elementary surface.

For any $\varepsilon > 0$, there exists $\gamma = \gamma(\varepsilon) > 0$ such that for each $(\vec{t}, x) \in B_\gamma \times \Su_{-\varepsilon}$, one has $\Phi_{\vec{t}}x \in \Uu$ and $\langle \alpha, \bs{\vec{Y}} \rangle (\Phi_{\vec{t}}x) \neq 0$ (here $B_\gamma = \{\vec{t} \in \R^p \colon \|\vec{t}\| < \gamma\}$).

Without loss of generality we may assume $\langle \alpha, \bs{\vec{Y}} \rangle (\Phi_{\vec{t}}x) > 0$. One has that the map $q: \Phi_{B\gamma}\Su_{-\varepsilon} \ni \Phi_{\vec{t}}x \mapsto x \in \Su_{-\varepsilon}$ is continuously differentiable.

Let $\Omega_\Su$ be a volume form on $\Su$; $\bs{X}$ a vector field on $\Su$; $\widetilde{\bs{X}}$ the vector field on $\Phi_{B_\gamma} \Su_{-\varepsilon}$ which is $q$-connected with $\bs{X}$ ($q_*(\widetilde{\bs{X}}(\Phi_{\vec{t}}x)) = \bs{X}(x)$); $\widetilde{\Omega} = q^*\Omega$ a differential $p$-form on $\Phi_{B_\gamma} \Su_{-\varepsilon}$.

Suppose that $\bs{\vec{X}} = \bs{X_1} \wedge \dots \wedge \bs{X_m}$ is a nowhere-vanishing multivector field on $\Su_{-\varepsilon}$ and let $\beta = \widetilde{\Omega} \wedge \alpha$. Then for $x \in \Su_{-\varepsilon}$,
\begin{equation*}
\langle \beta, \widetilde{\bs{\vec{X}}} \wedge \bs{\vec{Y}} \rangle (x) = \widetilde{\Omega} (\widetilde{\bs{\vec{X}}})(x) \cdot \alpha(\bs{\vec{Y}})(x) = (\Omega(\bs{\vec{X}}) \cdot \alpha(\bs{\vec{Y}}))(x) > 0.
\end{equation*}
(here we used $(i_{\bs{X_j}} \alpha)(x) = 0$). Choosing a smaller $\gamma > 0$ if needed, we conclude that $\beta$ is a volume form on $\Phi_{B_\gamma} \Su_{-\varepsilon} \subset \M$.

\begin{proposition}
\label{restriction_first_result}
Let $\bs{Z}$ be a vector field of class $C_b^1$ on $\Su$ and let $\dive_\Su \bs{Z}$ be the divergence of $\bs{Z}$ with respect to the volume form $\Omega$ on $\Su$. Given $\varepsilon > 0$, let $\widetilde{\bs{Z}}$ be the vector field on $\Phi_{B_\gamma} \Su_{-\varepsilon}$ which is $q$-connected with $\bs{Z}$ and let $\dive{\widetilde{\bs{Z}}}$ be the divergence of $\widetilde{\bs{Z}}$ with respect to the volume form $\beta$. Suppose that $\alpha$ is closed. Then 
\begin{equation}
\label{restriction_first_result_formula}
\dive_\Su \bs{Z} = (\dive{\widetilde{\bs{Z}}})\restrictTo{\Su}.
\end{equation}
\end{proposition}
\begin{proof}
Take $x \in \Su_{-\varepsilon}$. The statement follows from the following equalities
\begin{equation*}
(\dive{\widetilde{\bs{Z}}} \cdot \beta)(x) = (\ed i_{\widetilde{\bs{Z}}}(\widetilde{\Omega} \wedge \alpha))(x) = (\ed i_{\bs{Z}} \Omega)(x) \wedge \alpha(x) = (\dive_\Su \bs{Z} \cdot \beta)(x).
\end{equation*}
\end{proof}

\begin{corollary}
In the assumptions of Proposition \ref{restriction_first_result}, suppose that $\bs{\vec{Z}}$ is a multivector field of class $C_b^1$ on $\Su$; $\widetilde{\bs{\vec{Z}}}$ is the $q$-connected with $\bs{\vec{Z}}$ multivector field on $\mathcal{V} = \Phi_{B_\gamma} \Su_{-\varepsilon}$; $\dive_\Su$ and $\dive$ are the divergence operators on $(\Su, \Omega)$ and $(\mathcal{V}, \beta)$, respectively. Then
\begin{equation}
\label{restriction_formula_multivectors}
\dive_\Su \bs{\vec{Z}} = (\dive{\widetilde{\bs{\vec{Z}}}})\restrictTo{S}.
\end{equation}
\end{corollary}
\begin{proof}
Formula (\ref{restriction_formula_multivectors}) follows by induction from formula (\ref{restriction_first_result_formula}); recurrent formula (\ref{inductive_formula1}), applied to $\dive_\Su (\bs{X} \wedge \bs{\vec{Z}})$ and $\dive (\widetilde{\bs{X}} \wedge \widetilde{\bs{\vec{Z}}})$; equalities $\widetilde{\bs{X} \wedge \bs{\vec{Z}}} = \widetilde{\bs{X}} \wedge \widetilde{\bs{\vec{Z}}}$ and $\widetilde{\Lie_{\bs{X}} \bs{\vec{Z}}} = \Lie_{\widetilde{\bs{X}}} \widetilde{\bs{\vec{Z}}}$.
\end{proof}

Throughout the remainder of this article, $\M$ is a Banach manifold with a uniform atlas, modeled on a space $E$, where $E$ satisfies the assumptions of Theorem \ref{uniqueness_theorem}. Suppose that $\Su$ is an elementary surface in $\M$ of codimension $m$; $\mu$ is a (non-negative) Radon measure on $\M$ and the corresponding measure $\sigma = \sigma_{\bs{\vec{Y}}}$ on the surface $\Su_{-\varepsilon} \subset \Su$ is constructed as described in Section \ref{associated_measures}.

It follows from general theory of differential equations in Banach spaces that there exists $\gamma = \gamma(\varepsilon) > 0$ for which one has a well-defined map $q: \Phi_{B_\gamma} \Su_{-\varepsilon} \ni \Phi_{\vec{t}}x \mapsto x \in \Su_{-\varepsilon}$ of class $C_b^1$. Let $\bs{Z}$ be a vector field of class $C_b^1$ on $\Su$. Then the $q$-connected with $\bs{Z}$ vector field $\widetilde{\bs{Z}}$ is defined on $\mathcal{V} = \Phi_{B_\gamma} \Su_{-\varepsilon}$ and is also of class $C_b^1$.

\begin{theorem}
\label{theorem2}
Suppose that $\widetilde{\bs{Z}}$ has a divergence $\dive{\widetilde{\bs{Z}}} \in L_\infty(\mathcal{V}, \mu)$. Then $\bs{Z}$ has a divergence $\dive_\Su \bs{Z} \in L_\infty(\Su, \sigma)$ and for any bounded Borel function $u: \Su_{-\varepsilon} \to \R$, the following equality holds
\begin{equation}
\label{theorem2_eq}
\int\limits_{\Su_{-\varepsilon}} u \dive_\Su \bs{Z} \,d\sigma = \lim_{r \to 0} \frac{1}{\lambda_m(B_r)} \int\limits_{\Phi_{B_r}\Su_{-\varepsilon}} \widehat{u} \dive{\widetilde{\bs{Z}}} \,d\mu
\end{equation}
(here and henceforth $\widehat{u}(\Phi_{\vec{t}}x) = u(x)$ for $(\vec{t}, x) \in B_\gamma \times \Su_{-\varepsilon}$).
\end{theorem}
\begin{proof}
\textbf{Step 1}. Let $u \in C_0^1(\Su)$. Then $u \in C_0^1(\Su_{-\varepsilon})$ for some $\varepsilon > 0$. We shall prove that for any $r \in (0, \gamma)$, the following  holds
\begin{equation}
\label{theorem2_eq1}
\int\limits_{\Phi_{B_r}\Su_{-\varepsilon}} \widehat{u} \dive{\widetilde{\bs{Z}}} \,d\mu = - \int\limits_{\Phi_{B_r}\Su_{-\varepsilon}} \widetilde{\bs{Z}} \widehat{u} \,d\mu.
\end{equation}

The function $\widehat{u}$ is not of class $C_0^1(\mathcal{V})$. We will use the fact that $\widetilde{\bs{Z}}$ is tangent to each surface $\Phi_{\vec{t}}\Su_{-\varepsilon}$ for fixed $\vec{t} \in B_\gamma$.

Let us define a sequence of functions $\varphi_n \in C[0, r]$ for $n > 3$ as follows
\[\arraycolsep=0.5pt\def\arraystretch{1.5}
\varphi_n(s) = \displaystyle\left\{ \begin{array}{lll}
&0 &\text{ if } s \in \left[0, \frac{n - 3}{n}r \right] \cup \left[\frac{n - 1}{n}r, r \right], \\

&-\frac{n^2}{r^2}s + \frac{n(n - 3)}{r} &\text{ if } s \in \left[\frac{n - 3}{n}r, \frac{n - 2}{n}r \right], \\

&\frac{n^2}{r^2}s - \frac{n(n - 1)}{r} &\text{ if } s \in \left[\frac{n - 2}{n}r, \frac{n - 1}{n}r \right].
\end{array} \right.
\]
Then for the sequence of functions $h_n(s) = 1 + \int\limits_0^s \varphi_n(s) \,ds$, one has that the functions $u_n(\Phi_{\vec{t}}x) = h_n(\|\vec{t}\|) \cdot u(x)$ coincide with $\widehat{u}(\Phi_{\vec{t}}x)$ for $\|\vec{t}\| \leq \frac{n - 3}{n}r$, and $u_n \in C_0^1(\Phi_{B_r}\Su_\varepsilon)$.

Hence, we have
\begin{equation}
\label{theorem2_eq2}
\int\limits_{\Phi_{B_r}\Su_{-\varepsilon}} u_n \dive{\widetilde{\bs{Z}}} \,d\mu = - \int\limits_{\Phi_{B_r}\Su_{-\varepsilon}} \widetilde{\bs{Z}} u_n \,d\mu
\end{equation}
and
\begin{equation*}
(\widetilde{\bs{Z}} u_n)(\Phi_{\vec{t}}x) = h_n(\|\vec{t}\|) \cdot (\widetilde{\bs{Z}} \widehat{u})(\Phi_{\vec{t}}x) \text{ for } x \in \Su_{-\varepsilon}.
\end{equation*}
Passing in (\ref{theorem2_eq2}) to the limit as $n \to \infty$ we obtain (\ref{theorem2_eq1}).

Since the function $\widetilde{\bs{Z}} \widehat{u} \in C_b(\Phi_{B_\gamma}\Su_{-\varepsilon})$, Lemma \ref{lemma3} implies existence of the limit 
\begin{equation*}
\lim_{r \to 0} \frac{1}{\lambda_m(B_r)} \int\limits_{\Phi_{B_r}\Su_{-\varepsilon}} \widetilde{\bs{Z}} \widehat{u} \,d\mu = \int\limits_{\Su_{-\varepsilon}} \bs{Z} u \,d\sigma.
\end{equation*}
Therefore, using (\ref{theorem2_eq1}), we obtain the following equality
\begin{equation}
\label{theorem2_eq3}
\lim_{r \to 0} \frac{1}{\lambda_m(B_r)} \int\limits_{\Phi_{B_r} \Su_{-\varepsilon}} \widehat{u} \dive \widetilde{\bs{Z}} \,d\mu = - \int\limits_{\Su_{-\varepsilon}} \bs{Z} u \,d\sigma,
\end{equation}
that holds for any function $u \in C_0^1(\Su_{-\varepsilon})$.

\textbf{Step 2}. The model space $E_1$ of the manifold $\Su$ has a finite codimension in $E$ and therefore also admits a function of class $C^1(E_1)$ with a bounded non-empty support. The argument used in the proof of Theorem \ref{uniqueness_theorem} also
proves that the sets $U_\alpha = \{x \colon u_\alpha(x) > 0\}$, where $\{u_\alpha\} = C_0^1(\Su_{-\varepsilon})$, constitute a base of the topology of $\Su_{-\varepsilon}$.

Let $u \in \{u_\alpha\}$; $U = \{x \colon u(x) > 0\}$ is one of the sets of this base. Taking a sequence of smooth functions $h_n \in C^1(\R)$ that approximate the Heaviside step function $\chi$, we construct a sequence of functions $v_n = h_n \circ u$ for which $\{x \colon v_n(x) > 0\} = U$; $v_n \nearrow j_U = \chi \circ u$ (where $j_A$ denotes the indicator function of a set $A$) and $V_n = \{x \colon v_n(x) = 1\} \nearrow U$.

Nikod\'ym's theorem implies the uniform in $r \in (0, \gamma)$ convergence
\begin{equation*}
\sigma_r(U \setminus V_n) = \frac{1}{\lambda_m(B_r)} \mu(\Phi_{B_r}(U \setminus V_n)) \to 0, ~n \to \infty.
\end{equation*}
Since $\dive{\widetilde{\bs{Z}}} \in L_\infty(\mu)$, one also has a uniform in $r \in (0, \gamma)$ convergence
\begin{equation*}
\frac{1}{\lambda_m(B_r)} \int\limits_{\Phi_{B_r}\Su_{-\varepsilon}} \left|(\widehat{v_n} - \widehat{j_U}) \dive{\widetilde{\bs{Z}}} \right| \,d\mu \to 0, ~n \to \infty.
\end{equation*}

This uniform convergence and the convergence (\ref{theorem2_eq3}), together with the inequality
\begin{equation*}
\begin{split}
&\left| \frac{1}{\lambda_m(B_r)} \int\limits_{\Phi_{B_r}U} \dive{\widetilde{\bs{Z}}} \,d\mu - \frac{1}{\lambda_m(B_s)} \int\limits_{\Phi_{B_s}U} \dive{\widetilde{\bs{Z}}} \,d\mu \right| \\
& \leq \frac{1}{\lambda_m(B_r)} \int\limits_{\Phi_{B_r}\Su_{-\varepsilon}} \left|(\widehat{v_n} - \widehat{j_U}) \dive{\widetilde{\bs{Z}}} \right| \,d\mu \\
& + \frac{1}{\lambda_m(B_s)} \int\limits_{\Phi_{B_s}\Su_{-\varepsilon}} \left|(\widehat{v_n} - \widehat{j_U}) \dive{\widetilde{\bs{Z}}} \right| \,d\mu \\
& + \left| \frac{1}{\lambda_m(B_r)} \int\limits_{\Phi_{B_r}\Su_{-\varepsilon}} \widehat{v_n} \cdot \dive{\widetilde{\bs{Z}}} \,d\mu - \frac{1}{\lambda_m(B_s)} \int\limits_{\Phi_{B_s}\Su_{-\varepsilon}} \widehat{v_n} \cdot \dive{\widetilde{\bs{Z}}} \,d\mu \right|
\end{split}
\end{equation*}
allow us to conclude that the following limit exists
\begin{equation}
\label{theorem2_eq4}
\lim_{r \to 0} \frac{1}{\lambda_m(B_r)} \int\limits_{\Phi_{B_r}U} \dive{\widetilde{\bs{Z}}} \,d\mu.
\end{equation}

\textbf{Step 3}. Let $K$ be a compact subset of $\Su_{-\varepsilon}$. Then there is a sequence of sets $U_n \in \{U_\alpha\}$ such that $U_n \searrow K$.

Again, using Nikod\'ym's theorem and the fact that $\dive{\widetilde{\bs{Z}}} \in L_\infty(\mu)$, we obtain a uniform in $r \in (0, \gamma)$ convergence
\begin{equation*}
\lim_{n \to \infty} \frac{1}{\lambda_m(B_r)} \int\limits_{\Phi_{B_r}(U_n \setminus K)} \left| \dive{\widetilde{\bs{Z}}} \right| \,d\mu = 0.
\end{equation*}

From this uniform convergence and the convergence (\ref{theorem2_eq4}), together with the next inequality (here $r, s \in (0, \gamma)$)
\begin{equation*}
\begin{split}
&\left| \frac{1}{\lambda_m(B_r)} \int\limits_{\Phi_{B_r}K} \dive{\widetilde{\bs{Z}}} \,d\mu - \frac{1}{\lambda_m(B_s)} \int\limits_{\Phi_{B_s}K} \dive{\widetilde{\bs{Z}}} \,d\mu \right| \\
& \leq \frac{1}{\lambda_m(B_r)} \int\limits_{\Phi_{B_r}(U_n \setminus K)} \left|\dive{\widetilde{\bs{Z}}} \right| \,d\mu + \frac{1}{\lambda_m(B_s)} \int\limits_{\Phi_{B_s}(U_n \setminus K)} \left|\dive{\widetilde{\bs{Z}}} \right| \,d\mu \\
& + \left| \frac{1}{\lambda_m(B_r)} \int\limits_{\Phi_{B_r}U_n} \dive{\widetilde{\bs{Z}}} \,d\mu - \frac{1}{\lambda_m(B_s)} \int\limits_{\Phi_{B_s}U_n} \dive{\widetilde{\bs{Z}}} \,d\mu \right|
\end{split}
\end{equation*}
we conclude that the following limit exists
\begin{equation}
\label{theorem2_eq5}
\lim_{r \to 0} \frac{1}{\lambda_m(B_r)} \int\limits_{\Phi_{B_r}K} \dive{\widetilde{\bs{Z}}} \,d\mu.
\end{equation}

\textbf{Step 4}. Let $A$ be an arbitrary Borel subset of $\Su_{-\varepsilon}$. Let $K_n$ be a non-decreasing sequence of compact sets satisfying $\sigma(A \setminus K_n) < \frac{1}{n}$. Then for $C = \bigcap\limits_{n = 1}^\infty (A \setminus K_n)$, one has $\sigma(C) = 0$, and therefore
\begin{equation}
\label{theorem2_eq6}
\lim_{r \to 0} \frac{1}{\lambda_m(B_r)} \int\limits_{\Phi_{B_r}C} \left| \dive{\widetilde{\bs{Z}}} \right| \,d\mu = 0.
\end{equation}

Analogously to Step 3, we first obtain a uniform in $r \in (0, \gamma)$ convergence
\begin{equation*}
\lim_{n \to \infty} \frac{1}{\lambda_m(B_r)} \int\limits_{\Phi_{B_r}((A \setminus C) \setminus K_n)} \left| \dive{\widetilde{\bs{Z}}} \right| \,d\mu = 0,
\end{equation*}
and then use (\ref{theorem2_eq6}) and the existence of the limit (\ref{theorem2_eq5}) in order to conclude that the following limit exists
\begin{equation}
\label{theorem2_eq7}
\lim_{r \to 0} \frac{1}{\lambda_m(B_r)} \int\limits_{\Phi_{B_r}A} \dive{\widetilde{\bs{Z}}} \,d\mu.
\end{equation}

Let now $\tau_r$ denote the measure on $\B(\Su_{-\varepsilon})$ defined by
\begin{equation*}
\tau_r(A) \colonequals \frac{1}{\lambda_m(B_r)} \int\limits_{\Phi_{B_r}A} \dive{\widetilde{\bs{Z}}} \,d\mu.
\end{equation*} 
Existence of the limit (\ref{theorem2_eq7}) means that for each Borel set $A \in \B(\Su_{-\varepsilon})$, there exists a limit $\lim\limits_{r \to 0} \tau_r(A) \equalscolon \tau(A)$. Since $\dive{\widetilde{\bs{Z}}} \in L_\infty(\mu)$, the measure $\tau$ is absolutely continuous with respect to $\sigma$ and, additionally, $g_\varepsilon = \frac{d \tau}{d\sigma} \in L_\infty(\Su_{-\varepsilon}, \sigma)$ and
\begin{equation}
\label{estimate}
\|g_\varepsilon\|_{L_\infty(\sigma)} \leq \|\dive{\widetilde{\bs{Z}}}\|_{L_\infty(\mu)}.
\end{equation}

For any bounded Borel function $u$ on $\Su_{-\varepsilon}$, one has
\begin{equation}
\label{theorem2_eq8}
\lim\limits_{r \to 0} \frac{1}{\lambda_m(B_r)} \int\limits_{\Phi_{B_r}\Su_{-\varepsilon}} \widehat{u} \dive{\widetilde{\bs{Z}}} \,d\mu = \lim\limits_{r \to 0} \int\limits_{\Su_{-\varepsilon}} u \,d\tau_r = \int\limits_{\Su_{-\varepsilon}} u \cdot g_\varepsilon \,d\sigma.
\end{equation}

Since (\ref{theorem2_eq8}) holds for any bounded Borel function on $\Su_{-\varepsilon}$, it follows that $g_{\varepsilon_1} = g_{\varepsilon_2}\restrictTo{\Su_{-\varepsilon_1}}$ for $\varepsilon_2 \in (0, \varepsilon_1)$ and hence, there exists a Borel function $g$ defined on the whole of $\Su$, such that $g_\varepsilon = g\restrictTo{\Su_{-\varepsilon}}$ for any $\varepsilon > 0$; moreover, by (\ref{estimate}), $g \in L_\infty(\Su, \sigma)$.

In particular, by (\ref{theorem2_eq3}), for any function $u \in C_0^1(\Su)$, one has
\begin{equation*}
-\int\limits_\Su \bs{Z} u \,d\sigma = \int\limits_\Su u \cdot g \,d\sigma.
\end{equation*}
Therefore, there exists $\dive_\Su \bs{Z} = g$ on $\Su$; $\dive_\Su \bs{Z} \in L_\infty(\sigma)$, and for any bounded Borel function $u$ defined on $\Su_{-\varepsilon}$ for some $\varepsilon > 0$, equality (\ref{theorem2_eq}) holds. This completes the proof of the theorem.
\end{proof}

\begin{remark}
Analogously to Lemma \ref{lemma3}, one can prove that
\begin{equation*}
\int\limits_{\Su_{-\varepsilon}} u \dive_\Su \bs{Z} \,d\sigma = \lim\limits_{r \to 0} \frac{1}{\lambda_m(B_r)} \int\limits_{\Phi_{B_r} \Su_{-\varepsilon}} u \dive{\widetilde{\bs{Z}}} \,d\mu
\end{equation*}
for any function $u \in C_b(\M)$.
\end{remark}

For a differential $k$-form $\alpha$ of class $C_b^1$ on $\Su$, we define $\widehat{\alpha} \colonequals q^*\alpha$. For each $\varepsilon > 0$, the form $\widehat{\alpha}$ is defined on $\Phi_{B_{\gamma(\varepsilon)}}\Su_{-\varepsilon}$.

\begin{corollary}
In the assumptions of Theorem \ref{theorem2}, let $\bs{\vec{Z}} = \bs{Z_1} \wedge \dots \wedge \bs{Z_{k + 1}}$ be a decomposable multivector field of class $C_b^1$ on $\Su$. Given $\varepsilon > 0$, let $\widetilde{\bs{\vec{Z}}} = \widetilde{\bs{Z_1}} \wedge \dots \wedge \widetilde{\bs{Z_{k + 1}}}$ be the $q$-connected to $\bs{\vec{Z}}$ multivector field on $\Phi_{B_\gamma} \Su_{-\varepsilon}$, and suppose that for each $i \in \{1, \dots, k + 1\}$, there exists $\dive{\widetilde{\bs{Z_i}}} \in L_\infty(\mu)$. Then $\bs{\vec{Z}} \in D(\dive_\Su)$ and $\dive_\Su \bs{Z_i} \in L_\infty(\sigma)$ for each $i \in \{1, \dots, k + 1\}$. Moreover, for any $\varepsilon > 0$ and a differential $k$-form $\alpha$ of class $C_0^1(\Su)$, the following equality holds
\begin{equation*}
\int\limits_{\Su_{-\varepsilon}} \langle \alpha, \dive_\Su \bs{\vec{Z}} \rangle \,d\sigma = \lim\limits_{r \to 0} \frac{1}{\lambda_m(B_r)} \int\limits_{\Phi_{B_r}\Su_{-\varepsilon}} \langle \widehat{\alpha}, \dive{\widetilde{\bs{\vec{Z}}}} \rangle \,d\mu.
\end{equation*}
\end{corollary}
\begin{proof}
Induction on $k$. Theorem \ref{theorem2} constitutes the basis of the induction. The induction step is based on formula (\ref{inductive_formula2}).

Let $\bs{\vec{Z}} = \bs{X} \wedge \bs{\vec{Y}}$, where $\bs{\vec{Y}}$ is a $k$-vector field. Then $\widetilde{\bs{\vec{Z}}} = \widetilde{\bs{X}} \wedge \widetilde{\bs{\vec{Y}}}$ and $\langle \widehat{\alpha}, \dive{\widetilde{\bs{\vec{Z}}}} \rangle = \dive{\widetilde{\bs{X}}} \cdot \langle \widehat{\alpha}, \widetilde{\bs{\vec{Y}}} \rangle - \langle i_{\widetilde{\bs{X}}} \widehat{\alpha}, \dive{\widetilde{\bs{\vec{Y}}}} \rangle + \langle \widehat{\alpha}, \Lie_{\widetilde{\bs{X}}} \widetilde{\bs{\vec{Y}}} \rangle$.

Since $\langle \widehat{\alpha}, \widetilde{\bs{\vec{Y}}} \rangle = \widehat{\langle \alpha, \bs{\vec{Y}} \rangle}$, Theorem \ref{theorem2} implies that
\begin{equation*}
\int\limits_{\Su_{-\varepsilon}} \dive_\Su \bs{X} \cdot \langle \alpha, \bs{\vec{Y}} \rangle \,d\sigma = \lim\limits_{r \to 0} \frac{1}{\lambda_m(B_r)} \int\limits_{\Phi_{B_r}\Su_{-\varepsilon}} \dive{\widetilde{\bs{X}}} \cdot \langle \widehat{\alpha}, \widetilde{\bs{\vec{Y}}} \rangle \,d\mu.
\end{equation*}
Since one has $i_{\widetilde{\bs{X}}} \widehat{\alpha} = \widehat{i_{\bs{X}} \alpha}$, the equality
\begin{equation*}
\int\limits_{\Su_{-\varepsilon}} \langle i_{\bs{X}} \alpha, \dive_\Su \bs{\vec{Y}} \rangle \,d\sigma = \lim\limits_{r \to 0} \frac{1}{\lambda_m(B_r)} \int\limits_{\Phi_{B_r}\Su_{-\varepsilon}} \langle i_{\widetilde{\bs{X}}} \widehat{\alpha}, \dive{\widetilde{\bs{\vec{Y}}}} \rangle \,d\mu
\end{equation*}
follows from the induction hypothesis.

We have $\langle \widehat{\alpha}, \Lie_{\widetilde{\bs{X}}} \widetilde{\bs{\vec{Y}}} \rangle = \widehat{u}$, where $u = \langle \alpha, \Lie_{\bs{X}} \bs{\vec{Y}} \rangle$ is a function of class $C_b(\Su_{-\varepsilon})$, and therefore the equality
\begin{equation*}
\int\limits_{\Su_{-\varepsilon}} \langle \alpha, \Lie_{\bs{X}} \bs{\vec{Y}} \rangle \,d\sigma = \lim\limits_{r \to 0} \frac{1}{\lambda_m(B_r)} \int\limits_{\Phi_{B_r} \Su_{-\varepsilon}} \langle \widehat{\alpha}, \Lie_{\widetilde{\bs{X}}} \widetilde{\bs{\vec{Y}}} \rangle \,d\mu
\end{equation*}
is a direct consequence of Lemma \ref{lemma3}. 

Applying now formula (\ref{inductive_formula2}) to $\dive_\Su(\bs{X} \wedge \bs{\vec{Y}})$ we obtain the statement of the corollary. 
\end{proof}


\end{document}